\documentclass[11pt]{article}
\usepackage[utf8]{inputenc}
\usepackage[T1]{fontenc}
\usepackage{microtype}
\usepackage[a4paper]{geometry}
\usepackage{amsmath}
\usepackage{amsfonts}
\usepackage{amsthm}
\usepackage{dsfont}
\usepackage{amssymb}
\usepackage{txfonts}
\usepackage{authblk}

\theoremstyle{remark}
\newtheorem{rmk}{Remark}

\newtheorem{prop}{Proposition}[section]
\newtheorem{lem}[prop]{Lemma}
\newtheorem{thm}[prop]{Theorem}

\theoremstyle{definition}

\newtheorem{conj}[prop]{Conjecture}

\newcommand{\dep}{\mathfrak{D}}
\newcommand{\tr}{\operatorname{tr}}

\newcommand{\ind}{\mathds{1}}
\newcommand{\diff}{\mathrm{d}}
\newcommand{\normal}{\mathcal{N}}
\newcommand{\dto}{\rightsquigarrow}
\newcommand{\Xb}{\mathbb{X}}
\newcommand{\Yb}{\mathbb{Y}}
\usepackage{mathtools}

\DeclarePairedDelimiter\floor{\lfloor}{\rfloor}
\usepackage[backend=bibtex,style=authoryear,natbib=true, url=false,isbn=false, doi=false, eprint=false]{biblatex} 
\renewbibmacro{in:}{}

\usepackage{color}
\definecolor{darkraspberry}{rgb}{0.53, 0.15, 0.34}
\definecolor{britishracinggreen}{rgb}{0.0, 0.26, 0.15}
\definecolor{burntumber}{rgb}{0.54, 0.2, 0.14}
\definecolor{royalblue}{RGB}{0,78,156}

\begin{document}
\title{About limiting spectral distributions of block-rescaled empirical covariance matrices}
\author{Gilles Mordant\thanks{G. Mordant is thankful for the funding by the DFG for the SFB project 1456-A04. Enlightening discussions with J. Segers and T. Claeys during the redaction of this note are gratefully acknowledged.}}
\affil{\footnotesize Institut f\"ur Mathematische Stochastik, Universit\"at G\"ottingen, Goldschmidtstra\ss e 7, 37077 G\"ottingen}

\date{\today}
\maketitle
\begin{abstract}
We establish that the limiting spectral distribution of a block-rescaled empirical covariance matrix is an arcsine law when the ratio between the dimension and the underlying sample size converges to 1 and when the samples corresponding to each block are independent. We further propose a conjecture for the cases where the latter ratio converges to a constant in the unit interval.
\end{abstract}
\section*{Introduction}
Let $d = p+q$ and consider two structured $d \times d$-dimensional covariance matrices of the form
\[
	\Sigma = \begin{bmatrix} \Sigma_1 & \Psi \\ \Psi^\top & \Sigma_2 \end{bmatrix} \quad \text{ and } \quad \Sigma_0 = \begin{bmatrix} \Sigma_1 & 0 \\ 0 & \Sigma_2 \end{bmatrix},
\]
where $\Sigma_1$ is a $p\times p$ matrix. In what follows, a matrix with index zero always means that the off diagonal blocks of the matrix are set to zero.
Assuming $\Sigma_1$ and $\Sigma_2$ are invertible, now construct the standardized correlation matrix
\[
	R = \Sigma_0^{-1/2} \Sigma \Sigma_0^{-1/2}
	= \begin{bmatrix} I_p & \Upsilon \\ \Upsilon^\top & I_q \end{bmatrix}
	\qquad \text{where} \qquad
	\Upsilon = \Sigma_1^{-1/2} \Psi \Sigma_2^{-1/2}.
\]
This matrix is interesting for statistical purposes. Indeed,  applying the adjusted RV coefficient \citep{MS21} to this matrix, one would obtain the coefficient proposed in \citet[Equation~(8.10)]{yao2015sample}. Indeed, 
\[
\frac{1}{p}\tr \big( \Upsilon \Upsilon^\top\big) = \frac{1}{p}\tr \big( \Sigma_1^{-1/2} \Psi \Sigma_2^{-1/2} \Sigma_2^{-1/2} \Psi^\top \Sigma_1^{-1/2}\big)=\frac{1}{p}\tr \big(  \Psi \Sigma_2^{-1} \Psi^\top \Sigma_1^{-1}\big).
\] 
Further, since the diagonal blocks of the standardized matrix are identity matrices, the two dependence coefficients proposed in \citet{MS21} will give the same result, which is
\[
	\dep(R)
	= \frac{d_W^2\left(\Sigma_0^{-1/2} \Sigma \Sigma_0^{-1/2}, I_d\right)}%
	{\sup_{T \in \Gamma(I_p, I_q)} d_W^2\left(T, I_d\right)} = \frac{d_W^2\left(\Sigma_0^{-1/2} \Sigma \Sigma_0^{-1/2}, I_d\right)}%
	{\left(2 - \sqrt{2}\right) \min(p , q)},
\]
where $\Gamma(I_p, I_q)$ is the set of all $d \times d$ positive semi-definite matrices with diagonal blocks $I_p$ and $I_q$ and where $d_W$ is the Bures--Wasserstein distance.

\section{Spectra of block-rescaled empirical covariance matrices}

As is common in statistics, one is interested in studying empirical counterparts to the population quantities under scrutiny. 
Interestingly, $\dep(R)$ is a linear spectral statistic. Before studying the distribution of the empirical version of the latter coefficient, it is an interesting problem to understand the spectrum of the empirical counterpart of $\Sigma_0^{-1/2}\Sigma \Sigma_0^{-1/2}$ when both the dimension $d$ and the underlying sample size $n$ go to infinity. This is does not follow directly from classical results of random matrix theory as it here involves a particular construct with dependent sub-blocks. To the best of our knowledge, this hasn't been studied before. 

Before stating the main result of this paper, we need to set some notation. First, consider an i.i.d. sample of $d$-dimensional random vectors $Z_1, \ldots, Z_n$, where $Z_i \sim \normal(\mu , \Sigma)$ for $i \in \{1, \ldots, n\}$.
Further assume that  
\[
\Sigma = \begin{bmatrix} \Sigma_1 & 0 \\ 0& \Sigma_2\end{bmatrix},
\]
where $\Sigma_1, \Sigma_2$ are positive definite covariance matrices.
Based on this sample, consider 
\[
\Sigma_n := \frac{1}{n}  A, \qquad A:=\sum_{k=1}^n (Z_k- \bar{Z})(Z_k- \bar{Z})^\top,
\]
with $\bar{X}$ the sample mean.
To start with, assume that $d$ is even, set $p=q=d/2$ and extract the two $p\times p$ main diagonal blocks of $\Sigma_n$ to construct $\Sigma_{n,0}$. Let us then study the eigenvalues of the random matrix $\Sigma_n \Sigma_{n,0}^{-1}$. 
The matrix $A$ defined above is a Wishart matrix $W( \Sigma, n-1)$ and thus can be expressed as $\sum_{k=1}^{n-1} \tilde{Z}_k\tilde{Z}_k^\top$, a sum of $n-1$ i.i.d.\ vectors such that 
\[
\tilde{Z}_k=\begin{pmatrix}
X\\ Y
\end{pmatrix}_k \sim \normal(0, \Sigma).
\]
Define $\Yb_{p,n}$ (resp. $\Xb_{p,n}$), the $p\times (n-1)$ matrix obtained by stacking columnwise the $Y_k$ (resp. $X_k$) vectors.
We also have 
\[
\Xb_{p,n}\Xb_{p,n}^\top \overset{d}{=} \Sigma_{n,1},\quad \Yb_{p,n}\Yb_{p,n}^\top \overset{d}{=} \Sigma_{n,2},\quad \Xb_{p,n}\Yb_{p,n}^\top \overset{d}{=} \Psi_{n}.
\]
It is easy to see that studying the spectrum of $ \Sigma_n \Sigma_{n,0}^{-1}$ is indeed equivalent to studying the spectrum of $ \Sigma_{n,0}^{-1/2}\Sigma_n \Sigma_{n,0}^{-1/2}$ by matching the eigenvectors. 
Finally, note that $\Sigma_n \Sigma_{n,0}^{-1}$ is independent of the matrices $\Sigma_1$ and $\Sigma_2$ ans we can thus set the latter to $I_p$ and $I_q$ without loss of generality.

\subsection{Main result}
\label{sec: Main}
Our main result can be stated imprecisely: \textit{ 
In the case $d/n \approx 1$, for $n$ large and $d<n$,  the spectral law of $\Sigma_n \Sigma_{n,0}^{-1}$ is close to the one of an arcsine law with support $[0,2]$, under the assumption of equally sized and independent blocks. }

Recall that this law has the density 
\[
\frac{1}{\pi \sqrt{x(2-x)}} \ind_{\{x \in (0,2)\}}
\]
and its moments are given by 
\[
\int_0^2 \frac{x^k}{\pi \sqrt{x(2-x)}} \diff x= \frac{2^k \Gamma(k+\tfrac{1}{2})}{\sqrt{\pi} \Gamma(k+1) }. 
\]
The result stated above is made precise in the following theorem.
\begin{thm}
\label{thm: Main}
Consider two independent, infinite dimensional matrices $\Xb$ and $\Yb$ where all the entries each matrix are independent standard normal variables defined on the same probability space. Set $p$ a sequence of integers depending on $n$ such that $2p/n \nearrow 1$, as $n\to \infty$.
Obtain the sequences $\Yb_{p,n}$ and $\Xb_{p,n}$ by extracting the upper leftmost $p\times (n-1)$ (sub-)matrices of $\Xb$ and $\Yb$ and construct 
\[
\tilde{\Sigma}_n :=\begin{bmatrix} \Xb_{p,n}\Xb_{p,n}^\top & \Xb_{p,n}\Yb_{p,n}^\top \\ \Yb_{p,n}\Xb_{p,n}^\top & \Yb_{p,n}\Yb_{p,n}^\top  \end{bmatrix}.
\]
As before, construct $\tilde{\Sigma}_{n,0}$ as a copy of $\tilde{\Sigma}_n$ where the off-diagonal blocks are set to zero..
Then, the limiting spectral density of $\tilde{\Sigma}_n \tilde{\Sigma}_{n,0}^{-1}$ as $n \to \infty$ is 
\[
\frac{1}{\pi \sqrt{x(2-x)}} \ind_{\{x \in (0,2)\}} \quad a.s.
\] 
\end{thm}

\subsection{Further limiting distribution conjecture and open problems}
\begin{conj}
In the same setting as Thm~\ref{thm: Main}, when $2p/n \to c \in (0,1)$ as $n\to \infty$, the limiting distribution is a (generalised) Kesten--McKay law.
\end{conj}

This conjecture is supported by simulations. Using the method of moments for the cases considered in the conjecture turns out much more complicated and leads to apparently intractable expressions. Also, it might be interesting to find an alternative proof of Theorem~\ref{thm: Main} that is less computational and more insightful. 
Finally, we only considered the case where the two blocks were stemming from independent samples. What would happen in the case where the later are (linearly) dependent?
 
\subsection{Main result derivation}

First, let us show the following lemma. 
\begin{lem}
\label{lem: Spectrum}
 The spectrum of  $\Sigma \Sigma_0^{-1}$ is contained in [0, 2], irrespective of $p$. 
\end{lem}
\begin{proof}
Recall that the spectrum of $\Sigma \Sigma_0^{-1}$ is the same as the one of $ \Sigma_{n,0}^{-1/2}\Sigma_n \Sigma_{n,0}^{-1/2}$, which directly follows by matching the eigenbases appropriately. 
The matrix $ \Sigma_{n,0}^{-1/2}\Sigma_n \Sigma_{n,0}^{-1/2}$ is positive semi-definite and its diagonal blocks are $I_p$. The inequality in \citet{thompson1972inequalities} ensures that the largest eigenvalue cannot be larger than 2.
\end{proof}
 From Lemma~\ref{lem: Spectrum}, we have compactness of the support and we know that the limiting spectral distribution will be uniquely characterised by its moments, if it exists.
Our proof technique will thus be to compute the limits of the moments and show that they are equal to those conjectured.

We will make use of the following lemma that can be found in  \citet{bordenave2016spectrum}, for instance.
\begin{lem}
Let $\mu$ be uniquely characterised by its moments and $H_n$ a sequence of Hermitian matrices of size $n \times n$. Finally let $m_\ell$ be $\mu$'s $\ell$-th moment. If for all $\ell \in \mathbb{N}$, 
\[
\lim_{n \to \infty} \frac{1}{n} \tr H_n^\ell = m_\ell.
\]
Then,  
\[
\mu_{H_n}:= \frac{1}{n}\sum_{i=1}^n \delta_{\lambda_i(H_n)} \dto 
\mu,
\]
where $\dto$ denotes the weak convergence of measures. 
\end{lem}

We can now turn to the proof of the main result.

\begin{proof}[Proof of Theorem~\ref{thm: Main}]
Note that 
\begin{equation}
\label{eq: Decomp}
\Sigma \Sigma_0^{-1} = 
\begin{bmatrix} I_p & \Psi \Sigma_2^{-1} \\ \Psi^\top\Sigma_1^{-1} & I_q \end{bmatrix}=: I_d + C =: I_d +  \begin{bmatrix} 0 &B_1 \\  B_2  & 0 \end{bmatrix}.
\end{equation}
We adopt the trace formulation and are thus interested in the quantities 
\[
\frac{1}{d} \tr \left( \big( \tilde{\Sigma}_n \tilde{\Sigma}_{n,0}^{-1}\big)^k\right),\qquad k \in \mathbb{N}.
\]
Because of \eqref{eq: Decomp}, we have
 \[
\frac{1}{d} \tr \left( \big( \tilde{\Sigma}_n \tilde{\Sigma}_{n,0}^{-1}\big)^k\right) =\frac{1}{d} \tr \left( \big( I_d + C_n)^k\right),\qquad k \in \mathbb{N},
\]
where $C_n$ is the empirical counterpart to $C$. As the indicator matrix commutes with any matrix, the binomial formula can be used to deduce
\begin{equation}
\label{eq: LinkMomDepRV}
\frac{1}{d} \tr \left( \big( \tilde{\Sigma}_n \tilde{\Sigma}_{n,0}^{-1}\big)^k\right) = \frac{1}{d}\sum_{\ell=0}^k {k \choose \ell} \tr (C_n^\ell).
\end{equation}
Let us now have a look at traces of the powers of $C_n$. We have 
\[
\tr C^2= \tr \begin{bmatrix} B_{n,1}B_{n,2}  & 0 \\  0&B_{n,2} B_{n,1} 
\end{bmatrix}, 
\qquad \tr C^3= 0,
\qquad 
\tr C^4 = \tr \begin{bmatrix}     B_{n,1}B_{n,2} B_{n,1}B_{n,2}& 0\\  0&B_{n,2} B_{n,1}B_{n,2}B_{n,1}
\end{bmatrix}.
\]
From there, because of the properties of the matrices involved, we have 
\begin{equation}
\label{eq: MomRel}
\tr C_n^\ell = \begin{cases}
\ 2 \tr \left( \big( B_{n,1}B_{n,2} \big)^{\ell/2} \right), \quad &\ell \text{ even},\\
\quad 0, & k \text{ odd}.
\end{cases}
\end{equation}

Further, observe that \eqref{eq: MomRel} relates the moments of the eigenvalue distribution of $C$ with those of 
\begin{equation} 
\label{eq: M}
\Xb_{p,n}\Yb_{p,n}^\top \big(  \Yb_{p,n}\Yb_{p,n}^\top \big)^{-1}\Yb_{p,n}\Xb_{p,n}^\top \big(  \Yb_{p,n}\Yb_{p,n}^\top \big)^{-1}
\overset{d}{=} \Psi_n \Sigma_{n,2}^{-1} \Psi_n^\top \Sigma_{n,1}^{-1}, 
\end{equation} 
where the last equality in distribution is written to help the reader identify the elements.
\begin{rmk}
We have found a link between the dependence coefficients coming from the Bures-distance after ``block-rescaling'' standardisation and the modified RV coefficient again after ``block-rescaling'' standardisation.
\end{rmk}
To obtain the limiting distribution, we thus need to understand 
\begin{equation}
\label{eq: linkBPsi}
\tr \left( \big( B_{n,1}B_{n,2} \big)^{\ell/2} \right) \overset{d}{=}  \tr \left( \big( \Psi_n \Sigma_{n,2}^{-1} \Psi_n^\top \Sigma_{n,1}^{-1})^{\ell/2} \right), \quad \ell \text{ even}.
\end{equation}
Further, define the projector
\[
P_X:= \Xb_{p,n}^\top (\Xb_{p,n}\Xb_{p,n}^\top)^{-1}\Xb_{p,n}.
\]
An important fact in this regard, proved in \citep[Section 8.3.1]{yao2015sample}, is that \eqref{eq: M} can be rewritten as a function of a so-called Fisher matrix $F$. Indeed,
\begin{equation}
\label{eq: funcOfFish}
\Psi_n \Sigma_{n,2}^{-1} \Psi_n^\top \Sigma_{n,1}^{-1} \overset{d}{=} F ( F + \alpha_n I)^{-1}, 
\end{equation}
where 
\[
F :=  \frac{1}{p} \Yb_{p,n} P_X \Yb_{p,n}^\top \left( \frac{1}{n-1-p} \Yb_{p,n} (I_{n-1}-P_X)\Yb_{p,n}^\top \right)^{-1}
\]
and $\alpha_n := (n-1-p)/p$. 
The terminology Fisher matrix comes from an analogy with the F-test as, under our independence assumption, $\Yb_{p,n} P_X \Yb_{p,n}^\top$ and $\Yb_{p,n} (I_{n-1}-P_X)\Yb_{p,n}^\top$ are two independent Wishart matrices, see again \citep[Section 8.3.1]{yao2015sample}.

The spectral distribution of Fisher matrices has been well studied in the literature so that, combining \eqref{eq: linkBPsi} with \eqref{eq: funcOfFish}, it appears that establishing the limiting moments of the eigenvalue distribution of $\tilde{\Sigma}_n \tilde{\Sigma}_{n,0}^{-1}$ involves computing the integral
\[
\int_0^\infty \frac{x^k}{(x+s/t)^k } \diff F_{s,t}(x), 
\]
where $F_{s,t}$ is the Fisher Limiting Spectral Distribution. The equality $\alpha_n= s/t$ will be used later on. This distribution's density is given by 
\[
p_{s,t}(x):=\frac{1-t}{2\pi x (s+tx)} \ \sqrt{(b-x)(x-a)}, \quad a\le x\le b,
\]
with 
\[
a=a(s,t) = \frac{(1-h)^2}{(1-t)^2}, \quad b=b(s,t) = \frac{(1+h)^2}{(1-t)^2}, \quad h=h(s,t) = \sqrt{s+t -st}.
\]
By \citep[Theorem 2.23]{yao2015sample}, we have

\begin{align*}
I_k&:=\int_0^\infty \frac{x^k}{(x+s/t)^k } \diff F_{s,t}(x) \\
	&= - \frac{h^2 t ^{k}(1-t)}{4\pi i } 
	\oint_{\lvert z \rvert=1} \frac{(1+hz)^{k-1}(h+z)^{k-1} (1-z^2)^2}
	{z (h+tz)^{k+1} (t+hz)^{k+1}} \diff z \\
	&  = - \frac{ (1-t)}{4\pi  i  t  h^{k-1} } 
	\oint_{\lvert z \rvert=1} \frac{(1+hz)^{k-1}(h+z)^{k-1} (1-z^2)^2}
	{z (z+h/t)^{k+1} (z+t/h)^{k+1}} \diff z \\
	&= - \frac{ (1-t)}{2  t  h^{k-1} } 
	\left\{ \frac{(1+hz)^{k-1}(h+z)^{k-1} (1-z^2)^2}
	{ (z+h/t)^{k+1} (z+t/h)^{k+1}} \Bigg\vert_{z=0} + \frac{1}{k!} \frac{\partial^k}{\partial z^k}  \frac{(1+hz)^{k-1}(h+z)^{k-1} (1-z^2)^2}
	{ (z+h/t)^{k+1} z} \Bigg\vert_{z=-t/h}	\right\}\\
	&= - \frac{ (1-t)}{2  t  h^{k-1} } 
	\left\{ h^{k-1}  + \frac{1}{k!} \frac{\partial^k}{\partial z^k}  \frac{(1+hz)^{k-1}(h+z)^{k-1} (1-z^2)^2}
	{ (z+h/t)^{k+1} z} \Bigg\vert_{z=-t/h}	\right\} \\
\end{align*}
where Cauchy's theorem was used in the third equality. To avoid notations that are already cumbersome, we do not keep track in the notation that $I_k$ is a function of $s,t$.

It is possible to derive an explicit formula for the above quantity. Unfortunately, the latter appears rather intractable.
Still, one can compute that we have
\begin{equation}
\label{eq: firstMom}
I_1= \frac{t}{s+t}, \qquad I_2= \frac{t^2(s^2+s+t)}{(s+t)^3}.
\end{equation}

Note that the value of $I_1$ was already presented in \citet[Lemma 2.26]{yao2015sample}.
Still, in the setting that we are interested in, that is setting $s=q/p =1$, it turns out that $h=1$, which brings nice simplifications. 
We then have 
 \begin{align*}
 I_k &= - \frac{ (1-t)}{2  t   } 
	\left\{ 1  +  \frac{1}{k!} \frac{\partial^k}{\partial z^k}  \frac{(1+z)^{2k}(1-z)^2}
	{ (z+h/t)^{k+1} z} \Bigg\vert_{z=-t}	\right\}.
 \end{align*}
 Let us now develop the derivative and evaluate it for $z=-t$. We will use the shorthand notation ``$\partial^\ell$'' to mean   ``$\partial^\ell/\partial z^\ell$''. We will also use Pochammer's symbol for the rising factorial, i.e., 
 \[
 (x)_n := \frac{\Gamma(x+n)}{\Gamma(x)}.
\]  It is not difficult to check that 
 \begin{align*}
 &\partial^\ell  (1+z)^{2k}&&   = (1+z)^{2k-\ell} (2k+1-\ell)_\ell &&\\
  &\partial^\ell  (z-1)^{2}&&= (z-1)^{2-\ell} (3-\ell)_\ell&&\\
  & \partial^\ell \ z^{-1} && = (-1)^\ell z^{-1-\ell} \ell! &&\\
  & \partial^\ell  (z+1/t)^{-k-1} && = (z+1/t)^{-k-1-\ell} (-k-\ell)_\ell. &&\\
\end{align*}
We can then rewrite the derivative of the product using multinomial coefficients\footnote{The multinomial coefficient is defined by the equality ${k \choose j_1,j_2,j_3,j_4} = \frac{k!} { \prod_{i=1}^4 j_i! }$.}. Relying on the partial derivatives that were just computed, it holds that
\begin{align*}
& \frac{1}{k!}\frac{\partial^k}{\partial z^k}  \frac{(1+z)^{2k}(1-z)^2}{ (z+h/t)^{k+1} z}  \\ 
& \qquad= \frac{1}{k!}\sum_{j_1+j_2+j_3+j_4=k}{k \choose j_1,j_2,j_3,j_4}  \partial^{j_1}(1+z)^{2k} \partial^{j_2}(1-z)^2 \partial^{j_3} z^{-1} \partial^{j_4} (z+1/t)^{-k-1-\ell}  \\
& \qquad= \frac{1}{k!}\sum_{j_1+j_2+j_3+j_4=k}{k \choose j_1,j_2,j_3,j_4}  (1+z)^{2k-j_1} (2k+1-j_1)_{j_1} (z-1)^{2-j_2} (3-j_2)_{j_2} \\
& \qquad \qquad \qquad\qquad \times  \ (-1)^{j_3} z^{-1-j_3} j_3! (z+1/t)^{-k-1-j_4} (-k-j_4)_{j_4}. \\
\intertext{Using the fact that $(-x)_n= (-1)^n (x-n+1)_n$, }
& \frac{1}{k!}\frac{\partial^k}{\partial z^k}  \frac{(1+z)^{2k}(1-z)^2}{ (z+h/t)^{k+1} z}  \\ 
& \qquad=\sum_{j_1+j_2+j_3+j_4=k}\frac{ (-1)^{j_3+j_4}} { j_1 ! j_2 ! j_4 !}  (1+z)^{2k-j_1} (2k+1-j_1)_{j_1} (z-1)^{2-j_2} (3-j_2)_{j_2} \\
& \qquad \qquad\qquad \qquad \times \ z^{-1-j_3}  (z+1/t)^{-k-1-j_4} (k+1)_{j_4} \\ 
&\qquad=\sum_{j_1+j_2+j_3+j_4=k}  (-1)^{j_3+j_4} { 2k \choose j_1}{ 2 \choose j_2} { k+j_4 \choose k,j_4}(1+z)^{2k-j_1}  (z-1)^{2-j_2}  z^{-1-j_3}  (z+1/t)^{-k-1-j_4},  
\end{align*}
noticing that $(x+1-n)_n / n! = {x \choose n}$.
Setting $z=-t$ in the last sum of the previous display, we get 
\begin{align*}
& \sum_{j_1+j_2+j_3+j_4=k}  (-1)^{j_3+j_4} { 2k \choose j_1}{ 2 \choose j_2} { k+j_4 \choose k,j_4}(1-t)^{2k-j_1}  (-t-1)^{2-j_2}  (-t)^{-1-j_3}  (-t+1/t)^{-k-1-j_4}  \\
&= \sum_{j_1+j_2+j_3+j_4=k}  (-1)^{j_4  -j_2 +1} { 2k \choose j_1}{ 2 \choose j_2} { k+j_4 \choose k,j_4}(1-t)^{2k-j_1}  (t+1)^{2-j_2}  t^{-1-j_3}  (-t+1/t)^{-k-1-j_4} \\
&= \sum_{j_1+j_2+j_3+j_4=k}  (-1)^{j_4  -j_2 +1} { 2k \choose j_1}{ 2 \choose j_2} { k+j_4 \choose k,j_4}(1-t)^{2k-j_1}  (t+1)^{2-j_2}  t^{-j_3 + k  +j_4}  (1-t^2)^{-k-1-j_4} \\
&= \sum_{j_1+j_2+j_3+j_4=k}  (-1)^{j_4  -j_2 +1} { 2k \choose j_1}{ 2 \choose j_2} { k+j_4 \choose k,j_4}(1-t)^{k-j_1-1-j_4 }  (t+1)^{1-j_2 -k-j_4}  t^{-j_3 + k  +j_4}. 
\end{align*}
We thus have 
\begin{align*}
I_k= - \frac{(1-t)}{2t} - \frac{1}{2t}\sum_{j_1+j_2+j_3+j_4=k}  (-1)^{j_4  -j_2 +1} { 2k \choose j_1}{ 2 \choose j_2} { k+j_4 \choose k,j_4}(1-t)^{k-j_1-j_4 }  (t+1)^{1-j_2 -k-j_4}  t^{-j_3 + k  +j_4}.
\end{align*}

Recall that $s=q/p$ and set $t=q/(n-1-p)$ where the dependence of $p,q$ on $n$ is as before. Remark that $\alpha_n=s/t$, as required in the Fisher density because of \eqref{eq: funcOfFish}.
Now taking the limit as $n \to \infty$, under the same assumptions as above regarding $p,q,n$, we have that $t=q/(n-1-p)= p/(p-1) \to 1.$
Therefore, 
\begin{align*}
\lim_{n\to \infty}I_k 
&=  - \frac{1}{2}\sum_{\substack{j_1+j_2+j_3+j_4=k \\ j_1+ j_4 = k}}  (-1)^{j_4  -j_2 +1} { 2k \choose j_1}{ 2 \choose j_2} { k+j_4 \choose k,j_4}  2^{1-j_2 -k-j_4}  \\
&=  - \frac{1}{2}\sum_{\substack{ j_1+ j_4 = k}}  (-1)^{j_4  +1} { 2k \choose j_1}{ k+j_4 \choose k,j_4}  2^{1-k-j_4}\\
&=  - \frac{1}{2}\sum_{ j_1= 0}^k  (-1)^{k-j_1  +1} { 2k \choose j_1}{ 2k-j_1 \choose k,k-j_1}  2^{1-k-(k-j_1)}  \\
&= \sum_{ j_1= 0}^k  (-1)^{k-j_1 } { 2k \choose j_1}{ 2k-j_1 \choose k,k-j_1}  \left(\frac{1}{2}\right)^{2k - j_1}  \\
&=  \left(\frac{1}{2}\right)^{2k } \frac{(2k)!}{(k!)^2} \sum_{ j_1= 0}^k  { k \choose j_1} (-1)^{k-j_1 } 2^{ j_1}  \\
&=  \left(\frac{1}{2}\right)^{2k } \frac{(2k)!}{(k!)^2}.  \\
\end{align*}
We finally arrive at the fact that 
\begin{align*}
\label{eq: LinkMomDepRV}
\lim_{n\to \infty }\frac{1}{2p} \tr \left( \big( \Sigma_n \Sigma_{n,0}^{-1}\big)^k\right) 
&= \sum_{\ell=0}^k {k \choose \ell} \left(\frac{1}{2}\right)^{\ell} \frac{(\ell)!}{((\ell/2)!)^2} \ind\{\ell \text{ even}\} \\
&= \sum_{\eta=0}^{\floor*{k/2}} {k \choose 2\eta} {2\eta \choose \eta} \frac{1}{2^{2\eta}} \quad a.s.
\end{align*}
Finally, \citetitle{Mathematica} gives the following equality
\[ 
\sum_{\eta=0}^{\floor*{k/2}} {k \choose 2\eta} {2\eta \choose \eta} \frac{1}{4^{\eta}} =\frac{2^k (k-1/2)!}{\sqrt{\pi}  k! },
\]
which is precisely what we wanted; recall the moments of the arcsine distribution given in Section~\ref{sec: Main}. The proof of Theorem~\ref{thm: Main} is thus complete. \end{proof}

\printbibliography
\end{document}